\documentclass{cmslatex}
\usepackage[paperwidth=7in, paperheight=10in, margin=.875in]{geometry}
 \usepackage[backref,colorlinks,linkcolor=red,anchorcolor=green,citecolor=blue]{hyperref}
\usepackage{amsfonts,amssymb}
\usepackage{amsmath}
\usepackage{graphicx}
\usepackage{cite}
\usepackage{enumerate}
\renewcommand{\theequation}{\arabic{section}.\arabic{equation}}

   \allowdisplaybreaks
\begin{document}
 \title{On the continuous time limit of Ensemble Square Root Filters\thanks{December 01, 2020}}


          \author{Theresa Lange\thanks{Institut f\"ur Mathematik, Technische Universit\"at Berlin, Stra{\ss}e des 17. Juni 136, D-10623 Berlin (tlange@math.tu-berlin.de), \url{https://www.math.tu-berlin.de/fachgebiete_ag_stochfinanz/fg_mathematische_stochastik_stochastische_prozesse_in_den_neurowissenschaften/v_menue/mitarbeiter/theresa_lange/}}
          \and Wilhelm Stannat\thanks{Institut f\"ur Mathematik, Technische Universit\"at Berlin, Stra{\ss}e des 17. Juni 136, D-10623 Berlin and Bernstein Center for Computational Neuroscience, Philippstr. 13, D-10115 Berlin, Germany, (stannat@math.tu-berlin.de).}}

         \pagestyle{myheadings} \markboth{ON THE CONTINUOUS TIME LIMIT OF ESRF}{THERESA LANGE, WILHELM STANNAT} \maketitle

          \begin{abstract}
          We provide a continuous time limit analysis for the class of Ensemble Square Root Filter algorithms with deterministic model perturbations. In the particular linear case, we specify general conditions on the model perturbations implying convergence of the empirical mean and covariance matrix towards their respective counterparts of the Kalman-Bucy Filter. As a second main result we identify additional assumptions for the convergence of the whole ensemble towards solutions of the Ensemble Kalman-Bucy filtering equations introduced in [J. de Wiljes, S. Reich, W. Stannat, SIAM Journal on Applied Dynamical Systems, 17(2): 1152–-1181, 2018]. 
The latter result can be generalized to nonlinear Lipschitz-continuous model operators. A striking implication of our results is the fact that the limiting equations for the ensemble members are universal for a large class of Ensemble Square Root Filters. This yields a mathematically rigorous justification for the analysis of these algorithms with the help of the Ensemble Kalman-Bucy Filter.
               
          \end{abstract}

          \section{Introduction.}\label{intro}\\
          Consider the optimal filtering problem in continuous time which consists of estimating the current state of a diffusion process
\begin{equation}\label{contX}
{\rm d}X_t = f\left(X_t\right){\rm d}t + Q^{\frac{1}{2}}{\rm d}W_t, \hspace{0.5cm} X_t \in \mathbb{R}^{d},
\end{equation}
using observations
\begin{equation}\label{contY}
{\rm d}Y_t = g\left(X_t\right){\rm d}t + C^{\frac{1}{2}}{\rm d}V_t, Y_0 = 0,\hspace{0.5cm} Y_t \in \mathbb{R}^{p}.
\end{equation}
The processes $W$ and $V$ are independent standard Brownian motions, $Q$ and $C$ positive definite matrices, and $f$ and $g$ assumed to be Lipschitz-continuous. The solution of this problem is given by the posterior mean $\int x \pi_t({\rm d}x)$, where 
\begin{equation}
\pi_t({\rm d}x) = \mathbb{P}\left[ X_t \in {\rm d}x | \mathcal{Y}_t\right]
\end{equation}
is the conditional distribution of $X_t$ given $\mathcal{Y}_t := \{ Y_s : s \leq t\}$. In the last decades a hoard of algorithms has been proposed to specify or approximate $\pi_t$.\\
\noindent
In practice, however, observations are accumulated discretely in time rather than continuously. In a typical scenario, one observes a sequence of observations $Y_{t_k}, k =1,2,3,...$, where $t_{k+1} = t_k + h$ for some $h>0$, for which one solves the corresponding discrete-time filtering problem.\\
\noindent
In the particular case of linear operators $f(x) = Ax$ and $g(x) = Gx$, this problem can be solved with the famous Kalman Filter \cite{kalman1960} computing the mean $\bar{x}$ and covariance matrix $P$ of the in this case Gaussian $\pi$ according to the following recursion:
given a current estimate $\bar{x}_{k-1}^{a}$, $\bar{x}_{k-1}^{a}$ is propagated forward according to the system equation to yield the forecast $\bar{x}_k^{f}$. If at time $t_k$, an observation $Y_{t_k}$ is available, this will be used to update the current forecast to yield an improved estimate $\bar{x}_k^{a}$. When applied to the continuous-time setting, one uses (in the simplest case) the Euler-Maruyama time-discretization of (\ref{contX}) in the forecast step and updates each forecast using the observations in the form of $\Delta Y_k := Y_{t_k} - Y_{t_{k-1}}$. The precise evolution equations then read as follows:\\
Forecast:
\begin{align}
\bar{x}_k^{f} &= \bar{x}_{k-1}^{a} + hA\bar{x}_{k-1}^{a}\label{KFmeanF}\\
P_k^{f} &= \left({\rm Id} + hA\right)P_{k-1}^{a}\left({\rm Id} +hA\right)^T + Q\label{KFcovF}
\end{align}
Update:
\begin{align}
\bar{x}_k^{a} &= \bar{x}_k^{f} + K_k\left(\Delta Y_k - hG\bar{x}_k^{f}\right)\label{KFmeanA}\\
P_k^{a} &= \left({\rm Id} - hK_kG\right)P_k^{f}\label{KFcovA}\\
K_k &= P_k^{f}G^T\left(C + hGP_k^{f}G^T\right)^{-1}\label{gain}
\end{align}
It is well known that the Kalman Filter admits a continuous-time analogue, the Kalman-Bucy Filter \cite{kalman1961}, given by
\begin{align}
\bar{x}_t &= A\bar{x}_t{\rm d}t + P_tG^TC^{-1}\left({\rm d}Y_t - G\bar{x}_t{\rm d}t\right)\label{KBFmean}\\
{\rm d}P_t &= \left(AP_t + P_tA^T + Q - P_tG^TC^{-1}GP_t\right){\rm d}t.\label{Ricc}
\end{align}
In the nonlinear case, however, calculating the exact $\pi_t$ is in general not possible necessitating approximative schemes. Ensemble Kalman Filters (EnKF) form a class of second-order accurate Monte-Carlo algorithms approximating the conditional mean and covariance matrix with the help of the empirical mean and covariance matrix of an ensemble, and propagating the ensemble according to the nonlinear counterpart of the Kalman filtering equations. In the case of the popular stochastic EnKF (cf. \cite{evensen1994}, \cite{burgers1998}), for instance, each ensemble member is propagated according to
\begin{align}
X_{t_k}^{(i),f} &= X_{t_{k-1}}^{(i),a} + hf\left(X_{t_{k-1}}^{(i),a}\right) + Q^{\frac{1}{2}}\tilde{W}_k^{(i)},\label{StochEnKFF}\\
X_{t_k}^{(i),a} &= X_{t_k}^{(i),f} + K_k\left(\Delta Y_k^{(i)} - hg\left(X_{t_k}^{(i),f}\right)\right), \hspace{0.5cm} \Delta Y_k^{(i)} = \Delta Y_k + C^{\frac{1}{2}}\tilde{V}_k^{(i)},\label{StochEnKFA}
\end{align}
where $\tilde{W}_k^{(i)}, \tilde{V}_k^{(i)} \sim \mathcal{N}\left(0,h{\rm Id}\right)$ are independent samples.\\
\noindent
In our previous work \cite{langeStannat2019}, we were able to show the existence of a continuous time limit $h \to  0$ of (\ref{StochEnKFF}) and (\ref{StochEnKFA}) in the case of $f$ and $g$ being Lipschitz-continuous and bounded. Furthermore, we proved an even stonger convergence result in the case of a modified algorithm inspired by \cite{sakov2008}, replacing $\tilde{W}^{(i)}$ and $\tilde{V}^{(i)}$ by suitable deterministic perturbations. The filter in \cite{sakov2008} is a so called deterministic EnKF and the aim of this paper now is to generalize the latter result to the class of these filtering algorithms, in particular to the class of Ensemble Square Root Filters (ESRF).\\
\noindent
This class of algorithms has been introduced in order to replace the additional noise $\tilde{V}^{(i)}$ to the observations, used in the stochastic EnKF to avoid that the empirical covariance matrix underestimates the true error covariance (cf. \cite{burgers1998}). ESRF are widely used in the geosciences since they were shown to numerically perform better than their stochastic counterpart (see e.g. \cite{tippett2003}, \cite{sun2009}, \cite{nerger2012}). The most popular ESRF algorithms are the Ensemble Adjustment Kalman Filter (EAKF, see \cite{anderson2001}), the Ensemble Transform Kalman Filter (ETKF, see \cite{bishop2001}), and the unperturbed EnKF (Whitaker, Hamill (2002), see \cite{whitaker2002}), as summarized in the survey paper \cite{tippett2003}.\\
\noindent
The idea of the ESRF algorithms is the following: let $E_k^{f} := \left[X_{t_k}^{(i),f} - \bar{x}_k^{f}\right]_{i=1,...,M}$ denote the matrix of forecast deviations such that $P_k^{f} = \frac{1}{M-1}E_k^{f}\left(E_k^{f}\right)^T$ is the forecast covariance matrix. Then, in the case of linear observations, ESRF specify deterministic transformations of $E_k^{f}$ such that the resulting covariance matrix $P_k^{a}$ satisfies the Kalman equation (\ref{KFcovA}) exactly, i.e.
\begin{equation}\label{ESRFupdate}
P_k^{a} \overset{!}{=}\left({\rm Id} - hK_kG\right)P_k^{f}.
\end{equation}
\noindent
This paper is structured as follows: in the particular linear case and under appropriate assumptions on the deterministic model perturbations, we show in Section \ref{SectionConvCovMean} that the ensemble equations lead to closed recursion formulas for the empirical mean and covariance matrix which up to terms of order $h^2$ coincide with their respective Kalman filtering equations. Our main results concerning the convergence of the mean and covariance matrix towards their corresponding counterparts in the Kalman-Bucy Filter are summarized in Theorem \ref{covConvergenceESRF} and Theorem \ref{meanConvergenceESRF}.\\
\noindent
In Section \ref{SectionContLimit}, we then prove for the above three algorithms EAKF, ETKF, and Whitaker, Hamill (2002) the existence of the continuous time limit of the full ensemble $X_{t_k}^{(i),f}$ (respectively $X_{t_k}^{(i),a}$) towards the solution of the Ensemble Kalman-Bucy filtering equations (cf. \cite{bergemann2012}, \cite{reich2011}, \cite{deWiljes2018})
\begin{equation}\label{ContFilterAT}
{\rm d}X_t^{(i)} = AX_t^{(i)}{\rm d}t + Q^{\frac{1}{2}}\hat{W}_t^{(i)}{\rm d} t + P_tG^TC^{-1}\left({\rm d}Y_t - \frac{1}{2}G\left(X_t^{(i)}+\bar{x}_t\right){\rm d}t\right)
\end{equation}
as summarized in Theorem \ref{mainResult}. Our analysis can be generalized to the case of nonlinear, Lipschitz-continuous model operators $f$. It is a striking fact that this continuous-time equation shows up as a universal limit of a broad class of deterministic filtering algorithms. As will be discussed in Section \ref{SectionDiscuss}, this forms a powerful result in view of analyzing properties of the discrete-time counterparts.

\subsection{Notation.}\label{notation}
In the following, we will abbreviate 'deterministic EnKF with deterministic model perturbations' by 'fully deterministic EnKF'. For any vector $x \in \mathbb{R}^n$ and matrix $A \in \mathbb{R}^{n \times m}$ let $x^T$ resp. $A^T$ denote the respective transpose. Further let $\|A \|_F$ denote the Frobenius norm and $\|A\|$ the operator norm of a matrix $A$. Also for a quadratic matrix $A$, let tr$(A)$ denote the trace of $A$.\\
\noindent
In the subsequent analysis we will use the notation $x_t \lesssim y_t$ for $x_t \leq Cy_t$ for some constant $C > 0$ independent of $t$ (e.g. arising from the Cauchy-Schwarz inequality).\\
\noindent
For the ensemble $\left\{X^{(i)}, i,..., M\right\}$ of size $M$ let
\begin{equation}
\bar{x} := \frac{1}{M} \sum_{i=1}^M X^{(i)}
\end{equation}
denote the ensemble mean and
\begin{equation}
P := \frac{1}{M-1}\sum_{i=1}^M \left(X^{(i)} - \bar{x}\right)\left(X^{(i)}-\bar{x}\right)^T
\end{equation}
the ensemble covariance matrix. Further define
\begin{equation}\label{defV}
\mathcal{V} := \frac{1}{M-1} \sum_{i=1}^M \left\|X^{(i)} - \bar{x}\right\|^2 = \text{tr}(P).
\end{equation}
As mentioned above, the discrete-time algorithms are carried out for the partition $0 = t_0 < ... < t_L = T$ with $t_{k+1} = t_k + h$, $h > 0$. In that case, if $t \in [t_k, t_{k+1})$ we use the notation
\begin{equation}
\eta(t) := t_k, \hspace{0.5cm} \nu(t) := k, \hspace{1cm} \eta_+(t) := t_{k+1} \hspace{0.5cm} \nu_+(t) := k+1.
\end{equation}

\section{Deterministic model perturbations.}\label{SectionDetPerturb}\\
\noindent
As described in the introduction, the aim of ESRF algorithms is to transform the forecast ensemble in such a way that the covariance matrix of the resulting ensemble exactly satisfies (\ref{ESRFupdate}) of the Kalman Filter. The proposed transformations are deterministic in the sense that they do not introduce additional noise as in the case of the stochastic EnKF (see \eqref{StochEnKFA}). Observe that this way, ESRF algorithms less likely suffer from sampling errors in the analysis step. Inspired by the modified filter in \cite{deWiljes2018}, we reduce the total amount of noise in the algorithms even further by introducing deterministic model perturbations to replace the noise also in the forecast step. This ansatz can be interpreted as a form of inflation of the ensemble, a technique widely used in ensemble-based algorithms to increase the ensemble spread and prevent underestimation of the covariance matrix. Note that this way the only randomness present in the resulting algorithm comes from the initial ensemble and the observations.\\
In continuous time then, it seems to be natural to expect that the covariance of the continuous time limit, if it exists, satisfies (\ref{Ricc}) of the Kalman-Bucy Filter. Using (\ref{ESRFupdate}) and (\ref{StochEnKFF}) with deterministic model perturbations of the form $hQ^{\frac{1}{2}}\hat{W}_k^{(i),h}$ the evolution equations of the forecast and update covariance matrix read as follows:
\begin{align}\label{recurP}
P_k^{f} &= \left({\rm Id} +hA\right)P_{k-1}^{a}\left({\rm Id} +hA\right)^T\notag\\
&\hspace{0.5cm} + \frac{h}{M-1}\sum_{i=1}^M \left({\rm Id} + hA\right)^T\left(X_{t_{k-1}}^{(i),a} - \bar{x}_{k-1}^{a}\right)\left(\hat{W}_k^{(i),h} - \hat{w}_k^{h}\right)^TQ^{\frac{1}{2}}\notag\\
&\hspace{3cm} + Q^{\frac{1}{2}}\left(\hat{W}_k^{(i),h}-\hat{w}_k^{h}\right)\left(X_{t_{k-1}}^{(i),a} - \bar{x}_{k-1}^{a}\right)^T\left({\rm Id} + hA\right)^T\notag\\
&\hspace{0.5cm} + \frac{h^2}{M-1}\sum_{i=1}^M Q^{\frac{1}{2}}\left(\hat{W}_k^{(i),h} - \hat{w}_k^{h}\right)\left(\hat{W}_k^{(i),h} - \hat{w}_k^{h}\right)^TQ^{\frac{1}{2}},\\
P_k^{a} &= \left({\rm Id} - hK_kG\right)P_k^{f}.
\end{align}
Therefore finding a choice of $\hat{W}_k^{(i),h}, i=1,...,M,$ such that
\begin{equation}\label{Pkfrecur}
P_k^{f} = P_{k-1}^{f} + h\left(AP_{k-1}^{f} + P_{k-1}^{f}A^T + Q - K_{k-1}GP_{k-1}^{f}\right) + O(h^2)
\end{equation}
formally yields (\ref{Ricc}). Throughout the paper we assume the following properties of the model perturbations $\hat{W}_k^{(i),h}$:\\
{\bf Assumption 1.}\label{AssumDiscreteW}
{\em It holds uniformly in} $k = 1,..., L=\frac{T}{h}$
\begin{itemize}
\item $\frac{1}{M-1}\sum_{i=1}^M \left(X_{t_{k-1}}^{(i),a} - \bar{x}_{k-1}^{a}\right)\left(\hat{W}_k^{(i),h} - \hat{w}_k^{h}\right)^TQ^{\frac{1}{2}} = \frac{1}{2} Q$
\item $\left\|\frac{1}{M-1}\sum_{i=1}^M Q^{\frac{1}{2}}\left(\hat{W}_k^{(i),h} - \hat{w}_k^{h}\right)\left(\hat{W}_k^{(i),h}-\hat{w}_k^{h}\right)^TQ^{\frac{1}{2}}\right\|_F \leq \kappa$
\item {\em w.l.o.g. }$\hat{W}_k^{(i),h}$ {\em are centered, i.e. }$\hat{w}_k^{h} = 0$.
\end{itemize}
\noindent
One can easily check that the resulting $P_k^{f}$ satisfies the recursion (\ref{Pkfrecur}) and as will be shown in the next section, converges to a continuous-time matrix-valued process satisfying (\ref{Ricc}).
\begin{example}\label{exampleDiscrW}
In \cite{langeStannat2019}, we discussed a fully deterministic EnKF using perturbations of the form
\begin{equation}\label{ReichW}
\hat{W}_k^{(i),h} := \frac{1}{2}Q^{\frac{1}{2}}\left(P_{k-1}^{a}\right)^{-1}\left(X_{t_{k-1}}^{(i),a} - \bar{x}_{k-1}^{a}\right).
\end{equation}
These satisfy Assumption 1. Indeed: it holds
\begin{align}
&\frac{1}{M-1}\sum_{i=1}^M \left(X_{t_{k-1}}^{(i),a} - \bar{x}_{k-1}^{a}\right)\left(\hat{W}_k^{(i),h} - \hat{w}_k^{h}\right)^TQ^{\frac{1}{2}}\notag\\
&= \frac{1}{2}P_{k-1}^{a}\left(P_{k-1}^{a}\right)^{-1}Q^{\frac{1}{2}}Q^{\frac{1}{2}} = \frac{1}{2}Q.
\end{align}
Further it holds $\hat{w}_k^{h} = 0$ and
\begin{equation}
\frac{1}{M-1}\sum_{i=1}^M Q^{\frac{1}{2}}\left(\hat{W}_k^{(i),h} - \hat{w}_k^{h}\right)\left(\hat{W}_k^{(i),h}-\hat{w}_k^{h}\right)^TQ^{\frac{1}{2}} = \frac{1}{4}Q\left(P_{k-1}^{a}\right)^{-1}Q.
\end{equation}
Furthermore, $\left\|\left(P_k^{a}\right)^{-1}\right\|_F$ is bounded uniformly in $k$ (see Appendix \ref{appProofReichCov}).
\end{example}
\begin{remark}
One can replace Assumption 1 on the model perturbations $\hat{W}^{(i),h}$ by the following quadratic matrix equation
\begin{align}\label{alternativeW}
&\frac{1}{M-1} \sum_{i=1}^M \left({\rm Id} + hA\right)\left(X_{t_{k-1}}^{(i),a} - \bar{x}_{k-1}^{a}\right)\left(\hat{W}_k^{(i),h} - \hat{w}_k^{h}\right)^TQ^{\frac{1}{2}} \notag\\
&\hspace{2cm}+ Q^{\frac{1}{2}}\left(\hat{W}_k^{(i),h} - \hat{w}_k^{h}\right)\left(X_{t_{k-1}}^{(i),a} - \bar{x}_{k-1}^{a}\right)^T\left({\rm Id} + hA\right)^T \notag\\
&+ \frac{h}{M-1}\sum_{i=1}^{M} Q^{\frac{1}{2}}\left(\hat{W}_k^{(i),h} - \hat{w}_k^{h}\right)\left(\hat{W}_k^{(i),h}-\hat{w}_k^{h}\right)^TQ^{\frac{1}{2}}\notag\\
&= Q + h\tilde{R}_k
\end{align}
with rest term $\tilde{R}_k$ uniformly bounded in $k$. With the notation
\begin{align}
\tilde{E}_{k-1}^{a} &:= \frac{1}{\sqrt{M-1}}\left({\rm Id} + hA\right)\left[X_{t_{k-1}}^{(i),a} - \bar{x}_{k-1}^{a}\right]_{i=1,...,M},\\
\mathcal{W}_k &:= \frac{1}{\sqrt{M-1}}Q^{\frac{1}{2}}\left[\hat{W}_k^{(i),h}-\bar{w}_k^{h}\right]_{i=1,...,M} \label{WMatrix}
\end{align}
this yields the problem of solving
\begin{equation}\label{quad1}
\tilde{E}_{k-1}^{a}\mathcal{W}_k^T + \mathcal{W}_k\left(\tilde{E}_{k-1}^{a}\right)^T + h\mathcal{W}_k\mathcal{W}_k^T = Q + h\tilde{R}_k.
\end{equation}
In the particular case of $\tilde{R}_k \equiv 0$, if
\begin{equation}\label{AlternativeW}
\mathcal{W}_k = -\frac{1}{h}\tilde{E}_{k-1}^{a} \pm J_{k-1}
\end{equation}
where $J_{k-1}$ solves
\begin{equation}\label{quad2}
J_{k-1}J_{k-1}^T = \frac{1}{h^2}\left({\rm Id} + hA\right)P_{k-1}^{a}\left({\rm Id} + hA\right)^T + \frac{1}{h}Q,
\end{equation}
then solving (\ref{quad1}) reduces to solving (\ref{quad2}).\\
\noindent
Looking back on Example \ref{exampleDiscrW}, observe that (\ref{ReichW}) assumes $M \geq d+1$ necessary for $P^{a}$ to be invertible. In (\ref{AlternativeW}), $P^{a}$ need not have full rank since $Q$ is already assumed to do so. However, for $h$ small, the first term in (\ref{AlternativeW}) dominates and in case $P^{a}$ is rank-deficient may cause numerical instabilities. Thus the case of $M \leq d$ needs to be treated with great care. Assuming $M \leq d$ in Example \ref{exampleDiscrW}, the inverse  in (\ref{ReichW}) should be replaced by the generalized inverse $\left(P_{k-1}^{a}\right)^{\dagger}$ yielding
\begin{equation}
\hat{W}_k^{(i),h} := \frac{1}{2}Q^{\frac{1}{2}}\left(P_{k-1}^{a}\right)^{\dagger}\left(X_{t_{k-1}}^{(i),a} - \bar{x}_{k-1}^{a}\right)
\end{equation}
which gives in (\ref{recurP})
\begin{equation}
P_k^{f} = P_{k-1}^{a} + h\left( AP_{k-1}^{a} + P_{k-1}^{a}A^T + P_{k-1}^{a}\left(P_{k-1}^{a}\right)^{\dagger}Q + Q\left(P_{k-1}^{a}\right)^{\dagger}P_{k-1}^{a}\right)+ O(h^2).
\end{equation}
Motivated by this example, we presume that the quadratic matrix equation at the beginning of this remark should in general be replaced by imposing
\begin{equation}
\text{(\ref{alternativeW})} = \Pi_{k-1}Q +h\tilde{R}_k
\end{equation}
where $\Pi_{k-1}$ is the projection onto the span of the ensemble deviations such that
\begin{equation}
\begin{aligned}
&\Pi_{k-1} \left(X_{t_{k-1}}^{(i),a}-\bar{x}_{k-1}^{a}\right) = X_{t_{k-1}}^{(i),a} - \bar{x}_{k-1}^{a},\\
&\Pi_{k-1}v = 0 \hspace{0,5cm}\forall v \in \text{span}\left\{X_{t_{k-1}}^{(1),a} - \bar{x}_{k-1}^{a}, ..., X_{t_{k-1}}^{(M),a} - \bar{x}_{k-1}^{a}\right\}^{\perp}.
\end{aligned}
\end{equation}
Using stochastic perturbations instead of deterministic is another alternative in the case $M \leq d$ due to the regularizing effect of the noise. This shall, however, be discussed in a separate paper.
\end{remark}

\section{Continuous time limit I: mean and covariance matrix.}\label{SectionConvCovMean}\\
\noindent
Imposing Assumption 1 on the deterministic model perturbations $\hat{W}^{(i),h}$, we are now able to rigorously show that the resulting covariance process converges to a process $P$ satisfying the Riccati equation (\ref{Ricc}):
\begin{theorem}\label{covConvergenceESRF}
Let $P$ denote the continuous-time matrix-valued process satisfying (\ref{Ricc}) and let Assumption 1 hold. If
\begin{equation}
\left\|P_0 - P_0^{a}\right\| \in O(h)
\end{equation}
then
\begin{equation}\label{covConvergenceF}
\sup_{t\in [0,T]} \left\|P_t - P_{\nu(t)}^{f}\right\| \in O(h) \hspace{0.5cm}\text{ and }\hspace{0.5cm} \sup_{t\in [0,T]} \left\|P_t - P_{\nu(t)}^{a}\right\| \in O(h) .
\end{equation}
\end{theorem}
\noindent
First of all observe that $\|P_t\|$ is uniformly bounded in time on $[0,T]$. For the proof of Theorem \ref{covConvergenceESRF}, we further need the following lemma which is an immediate consequence of our above assumptions:
\begin{lemma}\label{boundPf}
Given Assumption 1, there exists a constant $0 < p_T^{*,f} < \infty$ depending on $\|P_0^{a}\|$ and $\kappa$ such that for all $k = 1,..., L$ it holds
\begin{equation}
\left\|P_k^{f}\right\| \leq p_T^{*,f}.
\end{equation}
This implies that also $\|K_k\|$ is uniformly bounded in $k$.
\end{lemma}\\
For the proof see Appendix \ref{Pfbound}.\\
\begin{proof}{\bf(Proof of Theorem \ref{covConvergenceESRF}.)}
It holds
\begin{align}
P_{t_k}-P_k^{f}&= P_0 - P_0^{f}\notag\\
&\hspace{0.5cm} + \int_0^{t_k} A\left(P_s - P_{\nu(s)}^{f}\right) + \left(P_s - P_{\nu(s)}^{f}\right)A^T\notag \\
&\hspace{1.5cm}- \left(P_sG^TC^{-1}GP_s - P_{\nu(s)}^{f}G^T\left(hGP_{\nu(s)}^{f}G^T+C\right)^{-1}GP_{\nu(s)}^{f}\right){\rm d}s\notag\\
&\hspace{0.5cm}+h^2 \sum_{j=0}^{k-1} R_j
\end{align}
where
\begin{align}
R_j &:= hAK_jGP_j^{f}A^T\notag\\
&\hspace{0.5cm} - \frac{1}{M-1}\sum_{i=1}^M Q^{\frac{1}{2}}\left(\hat{W}_{j+1}^{(i),h}-\hat{w}^{h}_{j+1}\right)\left(\hat{W}_{j+1}^{(i),h}-\hat{w}^{h}_{j+1}\right)^TQ^{\frac{T}{2}}\notag\\
&\hspace{0.5cm} - \left(AP_j^{f}A^T - AK_jGP_j^{f} - K_jGP_j^{f}A^T + AQ + QA^T\right).
\end{align}
By Lemma \ref{boundPf} and Assumption 1 we obtain a constant $0 < r_T^{*} < \infty$ such that $\|R_j\| < r_T^{*}$ for all $j = 0,..., L-1$. Further it holds
\begin{align}
&\left\|P_sG^TC^{-1}GP_s - P_{\nu(s)}^{f}G^T\left(hGP_{\nu(s)}^{f}G^T + C\right)^{-1}GP_{\nu(s)}^{f}\right\|\notag\\
&\leq \left\|P_sG^TC^{-1}GP_s - P_{\nu(s)}^{f}G^TC^{-1}GP_{\nu(s)}^{f}\right\| \notag\\
&\hspace{0.5cm}+ \left\|P_{\nu(s)}^{f}G^T\left(C^{-1} - \left(hGP_{\nu(s)}^{f}G^T + C\right)^{-1}\right)GP_{\nu(s)}^{f}\right\|\notag\\
&\leq \left\|P_s-P_{\nu(s)}^{f}\right\|\|G\|^2\|C^{-1}\|\left(\|P_s\|+\left\|P_{\nu(s)}^{f}\right\|\right)\notag\\
&\hspace{0.5cm}+ \left\|P_{\nu(s)}^{f}\right\|^2\|G\|^2\left\|C^{-1} - \left(hGP_{\nu(s)}^{f}G^T + C\right)^{-1}\right\|.
\end{align}
Using the Woodbury matrix identity, we can estimate
\begin{equation}
\left\|C^{-1} - \left(hGP_{\nu(s)}^{f}G^T + C\right)^{-1}\right\| \leq h\left\|C^{-1}\right\|^2\|G\|^2\left\|P_{\nu(s)}^{f}\right\|. 
\end{equation}
Thus, again using Lemma \ref{boundPf}, we obtain for a constant $\tilde{C} = \tilde{C}(T)$
\begin{equation}
\left\|P_{t_k}-P_k^{f}\right\| \leq \left\|P_0 - P_0^{a}\right\| + \tilde{C}\int_0^{t_k} \left\|P_s - P_{\nu(s)}^{f}\right\| {\rm d}s+h T r_T^{*}.
\end{equation}
Since $P_t$ satisfies the Riccati equation (\ref{Ricc}), one can further show that 
\begin{equation}
\sup_{t\in [0,T]} \left\|P_t - P_{\eta(t)}\right\| \in O(h)
\end{equation}
which by a Gronwall argument yields
\begin{equation}
\sup_{t\in [0,T]}\left\|P_t-P_{\nu(t)}^f\right\| \in O(h).
\end{equation}
Due to
\begin{equation}
P_k^{a} = \left({\rm Id} - hK_kG\right)P_k^{f},
\end{equation}
this further yields by Lemma \ref{boundPf}
\begin{equation}
\sup_{t\in [0,T]}\left\|P_t-P_{\nu(t)}^{a}\right\| \in O(h).
\end{equation}
\end{proof}
\\
\noindent
This convergence result further implies convergence of the ensemble mean: recall that by Assumption 1 it holds $\hat{w}_k^{h}=0$ which gives the following recursion:
\begin{align}
\bar{x}_k^{f} &= \left({\rm Id} + hA\right)\bar{x}_{k-1}^{a},\\
\bar{x}_k^{a} &= \bar{x}_k^{f} + K_k\left(\Delta Y_k - hG\bar{x}_k^{f}\right).
\end{align}
For the proceeding analysis in this section and throughout the whole paper it is important to stress the following: the Euler-Maruyama time-discretization of the observation process $Y$
\begin{equation}
\Delta Y_k := Y_{t_k} - Y_{t_{k-1}} \approx hGX_{t_{k-1}} + C^{\frac{1}{2}}\left(V_{t_k}-V_{t_{k-1}}\right)
\end{equation}
yields a discrete-time observation process with observation operator $hG$. Thus $hG$ will be the modeling assumption on the observations. The actual observations used in the update, however, take the form
\begin{equation}\label{formY}
\Delta Y_k = Y_{t_k} - Y_{t_{k-1}} = \int_{t_{k-1}}^{t_k} GX_s^{\text{ref}} {\rm d}s + C^{\frac{1}{2}}\left(V_{t_k} - V_{t_{k-1}}\right)
\end{equation}
where $X^{\text{ref}}$ is a reference trajectory of the continuous-time process $X$. We will assume that
\begin{equation}\label{RefTraj}
\sup_{t \in [0,T]} \mathbb{E}\left[\left\|X_t^{\text{ref}}\right\|^2\right] < \infty.
\end{equation}
Therefore, we will use the above approximate model of the observations to set up the filter but use (\ref{formY}) for the actual observations in the following analysis.\\
\noindent
First of all note that it holds:
\begin{lemma}\label{BoundMean}
The ensemble mean satisfies
\begin{equation}
\sup_{t \in [0,T]} \mathbb{E}\left[\left\|\bar{x}_{\nu(t)}^{a}\right\|^2\right] < \infty.
\end{equation}
\end{lemma}\\
For the proof see Appendix \ref{appProofBoundMean}.\\
\\
\noindent
This enables us to show:
\begin{theorem}\label{meanConvergenceESRF}
Let $\left(\bar{x}_t\right)_{t \in [0,T]}$ denote the continuous-time vector-valued process satisfying
\begin{equation}\label{KBmean}
{\rm d}\bar{x}_t = A\bar{x}_t{\rm d}t + P_tG^TC^{-1}\left({\rm d}Y_t - G\bar{x}_t{\rm d}t\right).
\end{equation}
Then
\begin{equation}
\mathbb{E}\left[\sup_{t \in [0,T]} \left\|\bar{x}_t - \bar{x}_{\nu(t)}^{a}\right\|^2\right] \in O(h).
\end{equation}
\end{theorem}
\begin{proof}
First of all, it holds
\begin{equation}
\mathbb{E}\left[\sup_{t \in [0,T]}\left\|\bar{x}_t - \bar{x}_{\nu(t)}^{a}\right\|^2\right] \lesssim \mathbb{E}\left[\sup_{t \in [0,T]}\left\|\bar{x}_t - \bar{x}_{\eta(t)}\right\|^2\right] + \mathbb{E}\left[\sup_{t \in [0,T]}\left\|\bar{x}_{\eta(t)} - \bar{x}_{\nu(t)}^{a}\right\|^2\right].
\end{equation}
The updated ensemble mean satisfies the following recursion
\begin{equation}
\bar{x}_k^{a} = \bar{x}_{k-1}^{a} + hA\bar{x}_{k-1}^{a} + K_k\left(\Delta Y_k - hG\bar{x}_{k-1}^{a} - h^2GA\bar{x}_{k-1}^{a}\right).
\end{equation}
Thus we obtain
\begin{align}
\bar{x}_{\eta(t)} - \bar{x}_{\nu(t)}^{a} &= \bar{x}_0 - \bar{x}_0^{a}\notag\\
&\hspace{0.5cm} + \int_0^{\eta(t)} A\left(\bar{x}_s - \bar{x}_{\nu(s)}^{a}\right) - \left(P_sG^TC^{-1}\bar{x}_s - K_{\nu_{+}(s)}G\bar{x}_{\nu(s)}^{a}\right) {\rm d}s\notag\\
&\hspace{0.5cm} + \int_0^{\eta(t)}\left(P_sG^TC^{-1} - K_{\nu_{+}(s)}\right){\rm d}Y_s\notag\\
&\hspace{0.5cm} - h\int_0^{\eta(t)} GA\bar{x}_{\nu(s)}^{a}{\rm d}s.
\end{align}
Using (\ref{formY}), we can estimate via the Cauchy-Schwarz inequality
\begin{align}
&\left\|\bar{x}_{\eta(t)} - \bar{x}_{\nu(t)}^{a}\right\|^2\notag\\
&\lesssim \left\|\bar{x}_0 - \Bar{x}_0^{a}\right\|^2\notag\\
&\hspace{0.5cm} + \eta(t) \int_0^{\eta(t)} \left(\|A\|^2+ \|P_s\|^2\|G\|^4\left\|C^{-1}\right\|^2\right)\left\|\bar{x}_s - \bar{x}_{\nu(s)}^{a}\right\|^2 \notag\\
&\hspace{3cm} + \left\|K_{\nu_{+}(s)} - P_sG^TC^{-1}\right\|^2\|G\|^2\left(\left\|X_s^{\text{ref}}\right\|^2 + \left\|\bar{x}_{\nu(s)}^{a}\right\|^2\right){\rm d}s\notag\\
&\hspace{0.5cm} + \left\|\int_0^{\eta(t)} \left(K_{\nu_{+}(s)} - P_sG^TC^{-1}\right)C^{\frac{1}{2}}{\rm d} V_s\right\|^2\notag\\
&\hspace{0.5cm} + h^2\eta(t)\int_0^{\eta(t)} \|G\|^2\|A\|^2\left\|\bar{x}_{\nu(s)}^{a}\right\|^2{\rm d}s
\end{align}
thus it holds
\newpage
\begin{align}
&\mathbb{E}\left[\sup_{t \in [0,T]} \left\|\bar{x}_{\eta(t)} - \bar{x}_{\nu(t)}^{a}\right\|^2\right]\notag\\
&\lesssim \mathbb{E}\left[\left\|\bar{x}_0 - \bar{x}_0^{a}\right\|^2\right]\notag\\
&\hspace{0.5cm} + T \int_0^T \mathbb{E}\left[\left(\|A\|^2+ \|P_s\|^2\|G\|^4\left\|C^{-1}\right\|^2\right)\sup_{r \in [0,s]} \left\|\bar{x}_r - \bar{x}_{\nu(r)}^{a}\right\|^2\right]\notag\\
&\hspace{3cm} + \mathbb{E}\left[\left\|K_{\nu_{+}(s)} - P_sG^TC^{-1}\right\|^2\|G\|^2\left(\left\|X_s^{\text{ref}}\right\|^2 + \left\|\bar{x}_{\nu(s)}^{a}\right\|^2\right)\right]{\rm d}s\notag\\
&\hspace{0.5cm} + \mathbb{E}\left[\sup_{t \in [0,T]} \left\|\int_0^{\eta(t)} \left(K_{\nu_{+}(s)} - P_sG^TC^{-1}\right)C^{\frac{1}{2}}{\rm d} V_s\right\|^2\right]\notag\\
&\hspace{0.5cm} + h^2T\int_0^{\eta(t)} \|G\|^2\|A\|^2\mathbb{E}\left[\left\|\bar{x}_{\nu(s)}^{a}\right\|^2\right]{\rm d}s.
\end{align}
\noindent
Observe now that we can estimate
\begin{align}
&\left\|K_{\nu_+(t)} - P_tG^TC^{-1}\right\|\notag\\
&\leq \left\|P_{\nu_+(t)}^{f}G^T\right\|\left\|\left(C+hGP_{\nu_+(t)}^{f}G^T\right)^{-1} - C^{-1}\right\| + \left\|P_{\nu_+(t)}^{f} - P_t\right\|\left\|G^TC^{-1}\right\|\notag\\
&\leq h \left\|P_{\nu_+(t)}^{f}\right\|^2\|G\|^3\|C^{-1}\|^2 + \left\|P_{\nu_+(t)}^{f} - P_t\right\|\|G\|\|C^{-1}\|.
\end{align}
Thus by Lemma \ref{boundPf} we obtain
\begin{equation}\label{GainEstimate}
\sup_{t\in [0,T]} \left\|K_{\nu_+(t)} - P_tG^TC^{-1}\right\| \in O(h).
\end{equation}
By assumption (\ref{RefTraj}), this yields
\begin{equation}
\mathbb{E}\left[\left\|K_{\nu_+(s)} - P_sG^TC^{-1}\right\|^2\|G\|^2\left\|X_s^{\text{ref}}\right\|^2\right] \in O(h^2)
\end{equation}
and we further deduce by the $L^p$-maximal-inequality
\begin{align}
&\mathbb{E}\left[\sup_{t \in [0,T]} \left\|\int_0^{\eta(t)} \left(K_{\nu_+(s)} - P_sG^TC^{-1}\right)C^{\frac{1}{2}}{\rm d}V_s\right\|^2\right] \notag\\
&\hspace{3cm}= \int_0^T \mathbb{E}\left[\left\|\left(K_{\nu_+(s)} - P_sG^TC^{-1}\right)C^{\frac{1}{2}}\right\|^2\right] {\rm d} s \in O(h^2).
\end{align}
Finally using (\ref{RefTraj}) and boundedness of $\|P_t\|$ on $[0,T]$, one can similarly show that
\begin{equation}
\mathbb{E}\left[\sup_{t \in [0,T]} \left\|\bar{x}_t - \bar{x}_{\eta(t)}\right\|^2\right] \in O(h).
\end{equation}
Thus Lemma \ref{BoundMean} and a Gronwall argument conclude the proof.
\end{proof}

\section{Algorithms.}\label{SectionAlgos}\\
\noindent
In this section, we introduce the three ESRF algorithms EAKF, ETKF, and the unperturbed filter from \cite{whitaker2002} which we will focus on in this paper.

\subsection{Ensemble Adjustment/Transform Kalman Filter.}
The transformations of $E_k^{f}$ in case of EAKF and ETKF are given by the following:
\begin{itemize}
\item EAKF: $E_k^{a} = A_k E_k^{f} $
\item ETKF: $E_k^{a} = E_k^{f} T_k$
\end{itemize}
for matrices $A_k$ and $T_k$ specified below. Then using the update step (\ref{KFmeanA}) of the ensemble mean, one computes the updated ensemble members via
\begin{equation}\label{ESRFmember}
X_{t_k}^{(i),a} = X_{t_k}^{(i),a} - \bar{x}_k^{a} + \bar{x}_k^{a} = E_k^{a} e_i + \bar{x}_k^{a}
\end{equation}
where $e_i \in \mathbb{R}^M$ with $(e_i)_j = \delta_{ij}, j= 1,..., M$. The structure of the transformation matrices and equivalence of both EAKF and ETKF has been summarized in \cite{tippett2003} in terms of the singular value decomposition factors of the forecast covariance matrix, as well as in the appendix of \cite{ott2004} using basic linear algebra.\\
In our recent paper \cite{langeStannat2020} using the following integral representation
\begin{equation}
\sqrt{P^{-1}} = \frac{1}{\sqrt{\pi}}\int_0^{\infty}\frac{1}{\sqrt{t}}e^{-tP}{\rm d}t
\end{equation}
for any symmetric positive semi-definite matrix $P$, we were able to identify an equivalent analytic representation of these transformations of the following form: as derived in \cite{langeStannat2020}, $A_k$ and $T_k$ are given by
\begin{align}
A_k &= \sqrt{P_k^{f}}\left({\rm Id} +h\sqrt{P_k^{f}}G^TC^{-1}G\sqrt{P_k^{f}}\right)^{-\frac{1}{2}}\sqrt{P_k^{f}}^{-1},\\
T_k &= \left({\rm Id} + \frac{h}{M-1}\left(E_k^{f}\right)^TG^TC^{-1}GE_k^{f}\right)^{-\frac{1}{2}}
\end{align}
where $\sqrt{P_k^{f}}$ denotes the symmetric positive semidefinite square root of $P_k^{f}$ and $\sqrt{P_k^{f}}^{-1}$ its pseudo inverse. These transformations are adjoint in the sense that it holds
\begin{equation}\label{equiAT}
A_k E_k^{f} = E_k^{f} T_k
\end{equation}
which is the consequence of the following important integral representation
\begin{equation}\label{IntegralRep}
A_kE_k^{f} = \frac{1}{\sqrt{\pi}}\int_0^{\infty}\frac{e^{-t}}{\sqrt{t}}e^{-thP_k^{f}G^TC^{-1}G}{\rm d}tE_k^{f} = E_k^{f}T_k
\end{equation}
(see \cite{langeStannat2020} for an in-depth discussion). We introduce the following first-order expansion of the integral term
\begin{equation}
\frac{1}{\sqrt{\pi}}\int_0^{\infty}\frac{e^{-t}}{\sqrt{t}}e^{-thP_k^{f}G^TC^{-1}G}{\rm d}t = {\rm Id} - \frac{h}{2}P_k^{f}G^TC^{-1}G + R_k^{h}.
\end{equation}
Then using \eqref{ESRFmember}, both EAKF and ETKF can be summarized as the following\\
{\bf Algorithm 1.}\label{EAKFETKF} (EAKF/ETKF)
\begin{align}
X_{t_k}^{(i),f} &= X_{t_{k-1}}^{(i),a} + hAX_{t_{k-1}}^{(i),a} + hQ^{\frac{1}{2}}\hat{W}_k^{(i),h},\\
X_{t_k}^{(i),a} &= X_{t_k}^{(i),f} - \frac{h}{2}P_k^{f}G^TC^{-1}GX_{t_k}^{(i),f} - h\left(K_k - \frac{1}{2}P_k^{f}G^TC^{-1}\right)G\bar{x}_k^{f} + K_k\Delta Y_k + R_k^{h}e_i.
\end{align}

\subsection{Whitaker, Hamill (2002).}
In \cite{whitaker2002}, the authors approach the problem of omitting stochastic perturbations in a different way: similar as in the Kalman Filter, the update step of the ensemble mean and the ensemble deviations should be of the form
\begin{align}
\bar{x}_k^{a} &= \bar{x}_k^{f} + K_k\left(\Delta Y_k - hG\bar{x}_k^{f}\right),\\
E_k^{a} &= E_k^{f} + \tilde{K}_k\left( \left(\Delta Y_k\right)^{'} - hGE_k^{f}\right)
\end{align}
with $K_k$ as in (\ref{gain}), where $\tilde{K}_k$ denotes the gain for the update of the ensemble deviations and in case of the stochastic EnKF, $\left(\Delta Y_k\right)^{'}$ denotes the perturbations added to the actual observation $\Delta Y_k$. In the aim of avoiding such perturbations, setting $\left(\Delta Y_k\right)^{'} = 0$ yields
\begin{equation}\label{updateWHdev}
E_k^{a} = \left({\rm Id} - h\tilde{K}_kG\right)E_k^{f}.
\end{equation}
This gives the correct covariance matrix only if $\tilde{K}_k$ solves
\begin{equation}\label{ansatzTildeK}
\left({\rm Id} - h\tilde{K}_kG\right)P_k^{f}\left({\rm Id} - h\tilde{K}_kG\right)^T \overset{!}{=} P_k^{a} = \left({\rm Id} - hK_kG\right)P_k^{f}
\end{equation}
which has the solution
\begin{equation}\label{tildeK}
\tilde{K}_k = P_k^{f}G^T\left(C+hGP_k^{f}G^T\right)^{-\frac{1}{2}}\left(\left(C+hGP_k^{f}G^T\right)^{\frac{1}{2}} + C^{\frac{1}{2}}\right)^{-1}.
\end{equation}
For ease of notation denote in the following
\begin{equation}
\tilde{C}_k:= \left(C+hGP_k^{f}G^T\right)^{-\frac{1}{2}}\left(\left(C+hGP_k^{f}G^T\right)^{\frac{1}{2}} + C^{\frac{1}{2}}\right)^{-1}.
\end{equation}
Thus using (\ref{ESRFmember}), we can deduce the following algorithm:\\
{\bf Algorithm 2.}\label{WH}(Whitaker, Hamill (2002))
\begin{align}
X_{t_k}^{(i),f} &= X_{t_{k-1}}^{(i),a} + hAX_{t_{k-1}}^{(i),a} + hQ^{\frac{1}{2}}\hat{W}_k^{(i),h},\\
X_{t_k}^{(i),a} &= X_{t_k}^{(i),f} +K_k\left(\Delta Y_k - hG\bar{x}_k^{f}\right) - h\tilde{K}_kG\left(X_{t_k}^{(i),f} - \bar{x}_k^{f}\right).
\end{align}

\subsection{Summary.}
Comparing Algorithm 1 and 2, we obtain a unified structure of the update step of the following form
\begin{equation}\label{UnifUpdate}
X_{t_k}^{(i),a} = X_{t_k}^{(i),f} - h\hat{K}_kGX_{t_k}^{(i),f} - h\left(K_k - \hat{K}_k\right)G\bar{x}_k^{f} + K_k \Delta Y_k + \mathcal{R}_k^{h}e_i
\end{equation}
where for EAKF/ETKF we denote
\begin{equation}\label{SummEATKF}
\hat{K}_k := \frac{1}{2}P_k^{f}G^TC^{-1}, \quad \mathcal{R}_k^{h} := R_k^{h}
\end{equation}
and for the unperturbed filter
\begin{equation}\label{SummWH}
\hat{K}_k := \tilde{K}_k, \quad \mathcal{R}_k^{h} = 0.
\end{equation}
\begin{lemma}\label{EstSumm}
The following estimates hold:
\begin{equation}\label{EstimateKSumm}
\left\|\hat{K}_k\right\|\leq \frac{1}{2}\left\|P_k^{f}\right\|\|G\|\left\|C^{-1}\right\|
\end{equation}
and
\begin{equation}\label{EATKF_R}
\left\|\mathcal{R}_k^{h}\right\| \leq \frac{3h^2}{8}\left\|P_k^{f}\right\|^2\left\|G^TC^{-1}G\right\|^2.
\end{equation}
\end{lemma}
\begin{proof}
On \eqref{EstimateKSumm}: for EAKF/ETKF the estimate is straightforward. For the unperturbed filter use
\begin{equation}
C^{\frac{1}{2}} \leq \left(C+hGP_k^{f}G^T\right)^{\frac{1}{2}}
\end{equation}
to deduce
\begin{equation}
\left\|\tilde{C}_k\right\|\leq \left\|\left(C+hGP_k^{f}G^T\right)^{-\frac{1}{2}}\right\|\left\|\left(\left(C+ h GP_k^{f}G^T\right)^{\frac{1}{2}} + C^{\frac{1}{2}}\right)^{-1}\right\| \leq \frac{1}{2}\left\|C^{-\frac{1}{2}}\right\|^2
\end{equation}
which gives the claim.\\
On \eqref{EATKF_R}: the claim trivially holds true in the case of the unperturbed filter. For EAKF/ETKF observe that with $\Theta:= G^TC^{-1}G$
\begin{align}
R_k^{h} &= \frac{1}{\sqrt{\pi}}\int_0^{\infty} \frac{e^{-t}}{\sqrt{t}}\left(e^{-thP_k^{f}\Theta}-{\rm Id} + thP_k^{f}\Theta\right){\rm d}t\\
&=\frac{1}{\sqrt{\pi}}\int_0^{\infty} \frac{e^{-t}}{\sqrt{t}}\left(-\int_0^te^{-shP_k^{f}\Theta}hP_k^{f}\Theta{\rm d}s +thP_k^{f}\Theta\right){\rm d}t\\
&=\frac{1}{\sqrt{\pi}}\int_0^{\infty} \frac{e^{-1}}{\sqrt{t}} \left(\int_0^t \int_0^s e^{-rhP_k^{f}\Theta} {\rm d}r {\rm d}s\right)\left(hP_k^{f}\Theta\right)^2{\rm d}t\\
&= \frac{1}{\sqrt{\pi}}\int_0^{\infty}\frac{e^{-t}}{\sqrt{t}}\left(\int_0^t\int_0^s \sqrt{P_k^{f}}e^{-rh\sqrt{P_k^{f}}\Theta\sqrt{P_k^{f}}}\sqrt{P_k^{f}} {\rm d}r {\rm d}s\right) h^2\Theta P_k^{f}\Theta{\rm d}t.
\end{align}
The claim then follows by using $\left\|e^{-th\sqrt{P_k^{f}}\Theta\sqrt{P_k^{f}}}\right\| \leq 1$.
\end{proof}

\subsection{Non-uniqueness of the transformations.}\label{RemNonUnique}
As has been pointed out in \cite{tippett2003}, the above transformations are not unique. Indeed, for an orthogonal matrix $\mathcal{U}_k$ note that, for instance, in case of the ETKF the modified transformation $\hat{E}_k := E_k^{f}T_k\mathcal{U}_k$ also yields the correct covariance matrix since
\begin{equation}
\hat{P}_k := \frac{1}{M-1}\hat{E}_k\hat{E}_k^T = \frac{1}{M-1}E_k^{f}T_k\mathcal{U}_k \mathcal{U}_k^T T_k^T \left(E_k^{f}\right)^T = \frac{1}{M-1}E_k^{f}T_kT_k^T \left(E_k^{f}\right)^T = P_k^{a}.
\end{equation}
Thus post-multiplying the transformed ensemble with an orthogonal matrix does not change the resulting covariance matrix. This issue was further exploited in \cite{wang2004} and \cite{livings2008}. In the latter, the authors elaborate more conditions on the matrix $\mathcal{U}_k$: for ${\bf 1} = (1,..., 1)^T$ note that it holds $E_k^{f} {\bf 1} = 0$. A transformation $\tau$ is called mean-preserving if after applying the transformation it still holds true $\tau\left(E_k^{f}\right){\bf 1} = 0$. This is a desirable property since it is needed in the update step (\ref{ESRFmember}). The EAKF, the filter by Whitaker and Hamill and, due to Lemma \ref{equiAT}, also the ETKF are mean-preserving. Thus an orthogonal post-multiplier $\mathcal{U}_k$ is appropriate in this sense, if it does not violate the mean-preserving property, i.e. satisfies $\tau\left(E_k^{f}\right)\mathcal{U}_k{\bf 1} = 0$. This is clearly the case if ${\bf 1}$ is an eigenvector of $\mathcal{U}_k$. In the following section, after conducting the continuous time limit analysis for the unmodified algorithms, we further elaborate on a possible extension to orthogonal transformations in Section \ref{nonUnique}.

\section{Continuous time limit II: ensemble members.}\label{SectionContLimit}\\
\noindent
Throughout this section, we assume that the deterministic model perturbations $\hat{W}^{(i),h}$ satisfy Assumption 1. First observe that formally taking the continuous time limit in Algorithm 1 then yields the following coupled system of differential equations with suitably defined model perturbations $\hat{W}_t^{(i)}$:
\begin{equation}\label{contFilterAT}
{\rm d}X_t^{(i)} = AX_t^{(i)}{\rm d}t + Q^{\frac{1}{2}}\hat{W}_t^{(i)}{\rm d} t + P_tG^TC^{-1}\left({\rm d}Y_t - \frac{1}{2}G\left(X_t^{(i)}+\bar{x}_t\right){\rm d}t\right).
\end{equation}
\begin{example}\label{exampleContW}
In \cite{langeStannat2019}, we were able to show a continuous time limit result for the case of perturbations of the form (\ref{ReichW}) yielding (\ref{contFilterAT}) with model perturbations
\begin{equation}\label{ReichWCont}
\hat{W}_t^{(i)} := \frac{1}{2}Q^{\frac{1}{2}}P_t^{-1}\left(X_t^{(i)} - \bar{x}_t\right).
\end{equation}
\end{example}
\noindent
Assuming that the processes $\left(\hat{W}_t^{(i)}\right)_{t \geq 0}$ are continuous and fulfill\\
{\bf Assumption 2.}\label{AssumContW}
{\em It holds:}
\begin{itemize}
\item $\frac{1}{M-1}\sum_{i=1}^M \left(X_t^{(i)} - \bar{x}_t\right)\left(\hat{W}_t^{(i)} - \hat{w}_t\right)^TQ^{\frac{1}{2}} = \frac{1}{2} Q$
\item {\em w.l.o.g. $\hat{W}_t^{(i)}$ are centered.}
\end{itemize}
\noindent
one can easily check that this yields the correct structure of first and second moment, i.e. corresponding to the processes (\ref{contFilterAT}), the ensemble mean $\bar{x}$  and the covariance matrix $P$ satisfy the Kalman-Bucy Filter equations (\ref{KBFmean}) and (\ref{Ricc}), respectively. Further we obtain that (\ref{ReichWCont}) is an exemplary choice of such perturbations.\\
\noindent
However, it is not clear whether Assumption 1 and Assumption 2 already yield convergence of the model perturbations. We therefore further impose\\
{\bf Assumption 3.}\label{approxW}
{\em There exists a constant $R_T > 0$ such that}
\begin{equation}
\sum_{i=1}^M \left\|\hat{W}_{\nu_{+}(t)}^{(i),h} - \hat{W}_t^{(i)}\right\|^2 \leq R_T\left(h^2 + \sum_{i=1}^M \left\|X_{\eta(t)}^{(i),a} - X_t^{(i)}\right\|^2\right).
\end{equation}
\begin{example}
Recall Examples \ref{exampleDiscrW} and \ref{exampleContW}. For these choices of model perturbation it holds
\begin{align}
&\sum_{i=1}^M \left\|\hat{W}_{\nu_{+}(t)}^{(i),h} - \hat{W}_t^{(i)}\right\|^2\notag\\
&= \sum_{i=1}^M \left\|\frac{1}{2}Q\left(\left(P_{\nu(t)}^{a}\right)^{-1}\left(X_{\eta(t)}^{(i),a} - \bar{x}_{\nu(t)}^{a}\right) - P_t^{-1}\left(X_t^{(i)}-\bar{x}_t\right)\right)\right\|^2\notag\\
&\lesssim \|Q\|^2\left((M-1)\left\|\left(P_{\nu(t)}^{a}\right)^{-1} - P_t^{-1}\right\|^2\mathcal{V}_t + \left\|\left(P_{\nu(t)}^{a}\right)^{-1}\right\|^2 \sum_{i=1}^M \left\|X_{\eta(t)}^{(i),a} - X_t^{(i)} \right\|^2\right)\notag\\
&\leq R_T\left(h^2 + \sum_{i=1}^M \left\|X_{\eta(t)}^{(i),a} - X_t^{(i)}\right\|^2\right)
\end{align}
by Theorem \ref{covConvergenceESRF}, Appendix \ref{appProofReichCov} and boundedness of $\left\|P_t^{-1}\right\|$ and $\mathcal{V}_t$ on $[0,T]$ (as shown in \cite{langeStannat2019}).
\end{example}

\noindent 
In the following, we assume that there exists a pathwise unique strong solution $X^{(i)}$ to (\ref{contFilterAT}) which is almost surely continuous and satisfies
\begin{equation}\label{boundContMem}
\sup_{t\in [0,T]} \mathbb{E}\left[\left\|X_t^{(i)}\right\|^2\right] < \infty.
\end{equation}
In case of Example \ref{exampleContW}, see the argumentation in \cite{deWiljes2018} on existence of such solutions. By using (\ref{formY}) and assumption (\ref{RefTraj}), we obtain that
\begin{equation}\label{contLimXCont}
\mathbb{E}\left[\sup_{t \in [0,T]} \sum_{i=1}^M \left\|X_t^{(i)} - X_{\eta(t)}^{(i)}\right\|^2\right] \in O(h).
\end{equation}
Then under the above assumptions, the main result of this paper now reads as follows:
\begin{theorem}\label{mainResult}
Consider Algorithm 1 or Algorithm 2, let $\left(X_t^{(i)}\right)_{t \in [0,T]}$ be the unique strong solution to (\ref{contFilterAT}) and let $\|P_0^{a}\|$ be bounded uniformly in $\omega$. If
\begin{equation}
\mathbb{E}\left[\sum_{i=1}^M\left\|X_0^{(i),a}-X_0^{(i)}\right\|^2\right] \in O(h),
\end{equation}
then it holds
\begin{equation}\label{senseOfConv}
 \mathbb{E}\left[\sup_{t \in [0,T]}\sum_{i=1}^M \left\|X_{\eta(t)}^{(i),a}-X_t^{(i)}\right\|^2\right] \in O(h).
\end{equation}
\end{theorem}
\medskip
\noindent
The proof of Theorem \ref{mainResult} will now be given in the following sections.
\subsection{Preliminaries.}\label{contLimPre}
\noindent
By Assumption 1 and using the same analysis as in \cite{deWiljes2018}, one can show that it holds (recall that $\hat{W}_k^{(i)}$ are centred)
\begin{align}\label{boundW}
\sum_{i=1}^M \left\|Q^{\frac{1}{2}}\hat{W}_k^{(i),h}\right\|^2 &\leq \sqrt{M}(M-1) \left\|\frac{1}{M-1}\sum_{i=1}^M Q^{\frac{1}{2}}\left(\hat{W}_k^{(i),h}-\hat{w}_k^{h}\right)\left(\hat{W}_k^{(i),h}-\hat{w}_k^{h}\right)^TQ^{\frac{1}{2}}\right\|_F \notag\\
&\leq \sqrt{M}(M-1)\kappa.
\end{align}
With the above, we obtain the following result:
\begin{lemma}\label{boundMemAT}
For both Algorithm 1 and Algorithm 2 it holds
\begin{equation}
\sup_{t \in [0,T]} \mathbb{E}\left[\sum_{i=1}^M \left\|X_{\eta(t)}^{(i),a}\right\|^2\right] < \infty.
\end{equation}
\end{lemma}\\
For the proof see Appendix \ref{appProofBoundMem}.

\subsection{The continuous time limit.}\label{secContLim}
\noindent
Using \eqref{UnifUpdate}, observe that the difference process satisfies
\begin{align}
&X_{\eta(t)}^{(i),a} - X_{\eta(t)}^{(i)} \\
&= X_0^{(i),a} - X_0^{(i)}\\
&\hspace{0.5cm} + \int_0^{\eta(t)} A\left(X_{\eta(s)}^{(i),a} - X_s^{(i)}\right) + Q^{\frac{1}{2}}\left(\hat{W}_{\nu_+(s)}^{(i),h} - \hat{W}_s^{(i)}\right)\notag\\
&\hspace{1cm} + \left(K_{\nu_+(s)} - P_sG^TC^{-1}\right)GX_s^{\text{ref}}\\
&\hspace{1cm} - \left(\hat{K}_{\nu_{+}(s)}GX_{\eta_{+}(s)}^{(i),f} + \left(K_{\nu_{+}(s)}-\hat{K}_{\nu_{+}(s)}\right)G\bar{x}_{\nu_{+}(s)} - \frac{1}{2}P_sG^TC^{-1}G\left(X_s^{(i)}+\bar{x}_s\right)\right){\rm d}s\notag\\
&\hspace{0.5cm} + \int_0^{\eta(t)} \left(K_{\nu_+(s)} - P_sG^TC^{-1}\right)C^{\frac{1}{2}}{\rm d}V_s +  \mathcal{R}_{\nu(t)}^{h}.
\end{align}
With
\begin{align*}
&\hat{K}_{\nu_{+}(s)}GX_{\eta_{+}(s)}^{(i),f} + \left(K_{\nu_{+}(s)}-\hat{K}_{\nu_{+}(s)}\right)G\bar{x}_{\nu_{+}(s)} - \frac{1}{2}P_sG^TC^{-1}G\left(X_s^{(i)}+\bar{x}_s\right)\\
&= \hat{K}_{\nu_{+}(s)}G\left(X_{\eta_{+}(s)}^{(i),f} - X_s^{(i)}\right) + \left(\hat{K}_{\nu_{+}(s)}-\frac{1}{2}P_sG^TC^{-1}\right)GX_s^{(i)} \\
&\hspace{0.5cm}+ \left(K_{\nu_{+}(s)}-\hat{K}_{\nu_{+}(s)}\right)G\left(\bar{x}_{\nu_{+}(s)}^{f}-\bar{x}_s\right)+\left(K_{\nu_{+}(s)}-\hat{K}_{\nu_{+}(s)}-\frac{1}{2}P_sG^TC^{-1}\right)G\bar{x}_s
\end{align*}
and the Cauchy-Schwarz-inequality we estimate
\begin{align}
&\sum_{i=1}^M \left\|X_{\eta(t)}^{(i),a} - X_{\eta(t)}^{(i)}\right\|^2\notag\\
&\lesssim \sum_{i=1}^M \left\|X_0^{(i),a} - X_0^{(i)}\right\|^2\notag\\
&\hspace{0.5cm} + \eta(t)\int_0^{\eta(t)} \|A\|^2 \sum_{i=1}^M \left\|X_{\eta(s)}^{(i),a} - X_s^{(i)}\right\|^2 + \|Q\| \sum_{i=1}^M \left\|\hat{W}_{\nu_{+}(s)}^{(i)} - \hat{W}_s^{(i)}\right\|^2\notag\\
&\hspace{2cm} + M\left\|K_{\nu_{+}(s)}-P_sG^TC^{-1}\right\|^2\|G\|^2\left\|X_s^{\text{ref}}\right\|^2\notag\\
&\hspace{2cm} + \left(\left\|\hat{K}_{\nu_{+}(s)}\right\|^2+\left\|K_{\nu_{+}(s)}-\hat{K}_{\nu_{+}(s)}\right\|^2\right)\|G\|^2\sum_{i=1}^M\left\|X_{\eta_{+}(s)}^{(i),f}-X_s^{(i),s}\right\|^2 \notag\\
&\hspace{2cm}+ \left(\left\|\hat{K}_{\nu_{+}(s)}-\frac{1}{2}P_sG^TC^{-1}\right\|^2 \right.\notag\\
&\hspace{3cm}\left.+ \left\|K_{\nu_{+}(s)}-\hat{K}_{\nu_{+}(s)}-\frac{1}{2}P_sG^TC^{-1}\right\|^2\right)\|G\|^2\sum_{i=1}^{\infty} \left\|X_s^{(i)}\right\|^2{\rm d}s\notag\\
&\hspace{0.5cm} + M\left\|\int_0^{\eta(t)}\left(K_{\nu_{+}(s)}-P_sG^TC^{-1}\right)C^{\frac{1}{2}}{\rm d}V_s\right\|^2 + M\left\|\mathcal{R}_{\nu(t)}^{h}\right\|^2
\end{align}
\noindent
Recall from the proof of Theorem \ref{meanConvergenceESRF} that it holds
\begin{equation}
\mathbb{E}\left[\left\|K_{\nu_+(s)} - P_sG^TC^{-1}\right\|^2\|G\|^2\left\|X_s^{\text{ref}}\right\|^2\right] \in O(h^2)
\end{equation}
and
\begin{equation}
\mathbb{E}\left[\sup_{t \in [0,T]} \left\|\int_0^{\eta(t)} \left(K_{\nu_+(s)} - P_sG^TC^{-1}\right)C^{\frac{1}{2}}{\rm d}V_s\right\|^2\right] \in O(h^2).
\end{equation}
\noindent
Furthermore by Lemma \ref{boundPf} and Lemma \ref{EstSumm}, $\left\|K_k\right\|$ and $\left\|\hat{K}_k\right\|$ are uniformly bounded in $k$, and again from Lemma \ref{EstSumm} we obtain
\begin{equation}
\left\|\mathcal{R}_k^{h}\right\| \in O(h^2)
\end{equation}
uniformly in $k$. It remains to investigate the differences of the Kalman gains. We claim that it holds
\begin{equation}
\sup_{t \in [0,T]}\left\|\hat{K}_{\nu_{+}(t)} - \frac{1}{2}P_tG^TC^{-1}\right\| \in O(h)
\end{equation}
and
\begin{equation}
\sup_{t \in [0,T]}\left\|K_{\nu_{+}(t)}-\hat{K}_{\nu_{+}(t)}-\frac{1}{2}P_tG^TC^{-1}\right\| \in O(h).
\end{equation}
Indeed: in the case of EAKF/ETKF, the claim follows from the decompositions
\begin{equation}
\hat{K}_{\nu_{+}(t)} - \frac{1}{2}P_tG^TC^{-1} = \frac{1}{2}\left(P_{\nu_{+}(t)}-P_t\right)G^TC^{-1}
\end{equation}
and
\begin{align}
&K_{\nu_{+}(t)}-\hat{K}_{\nu_{+}(t)}-\frac{1}{2}P_tG^TC^{-1}\notag\\
&= P_{\nu_{+}(t)}^{f}G^T\left(\left(C+hGP_{\nu_{+}(t)}^{f}G^T\right)^{-1}-C^{-1}\right) + \frac{1}{2}\left(P_{\nu_{+}(t)}^{f}-P_t\right)G^TC^{-1}
\end{align}
together with Lemma \ref{boundPf} and Theorem \ref{covConvergenceESRF}. For the unperturbed filter we decompose the differences in such a way that
\begin{align}
&\hat{K}_{\nu_{+}(t)} - \frac{1}{2}P_tG^TC^{-1}\notag\\
&= \frac{1}{2}\left(P_{\nu_{+}(t)}^{f}-P_t\right)G^TC^{-1}+ P_{\nu_{+}(t)}^{f}G^T\left(\tilde{C}_{\nu_{+}(t)} - \frac{1}{2}C^{-1}\right)\notag\\
&= \frac{1}{2}\left(P_{\nu_{+}(t)}^{f}-P_t\right)G^TC^{-1}\notag\\
&\hspace{0.5cm} + \tilde{C}_{\nu_{+}(t)}\left( C^{\frac{1}{2}}\left(C^{\frac{1}{2}} - \left(C+hGP_{\nu_{+}(t)}^{f}G^T\right)^{\frac{1}{2}}\right) - hGP_{\nu_{+}(t)}^{f}G^T\right)\frac{1}{2}C^{-1}
\end{align}
and
\begin{align}
&K_{\nu_+(t)} - \hat{K}_{\nu_+(t)}\\
&= P_{\nu_+(t)}^{f} G^T\left(C+hGP_{\nu_+(t)}^{f}G^T\right)^{-\frac{1}{2}}\notag\\
&\hspace{0.5cm}\times\left( \left(C+hGP_{\nu_+(t)}^{f}G^T\right)^{-\frac{1}{2}} - \left( \left(C+hGP_{\nu_+(t)}^{f}G^T\right)^{\frac{1}{2}} + C^{\frac{1}{2}}\right)^{-1}\right)\notag\\
&= K_{\nu_+(t)}C^{\frac{1}{2}}\left( \left(C+hGP_{\nu_+(t)}^{f}G^T\right)^{\frac{1}{2}} + C^{\frac{1}{2}}\right)^{-1}
\end{align}
which gives
\begin{align}
&K_{\nu_+(t)} - \hat{K}_{\nu_+(t)} - \frac{1}{2}P_tG^TC^{-1}\notag\\
&= \left(K_{\nu_+(t)} - P_tG^TC^{-1}\right)C^{\frac{1}{2}}\left( \left(C+hGP_{\nu_+(t)}^{f}G^T\right)^{\frac{1}{2}} + C^{\frac{1}{2}}\right)^{-1}\notag\\
&\hspace{0.5cm} + P_tG^TC^{-1}\left( C^{\frac{1}{2}}\left( \left(C+hGP_{\nu_+(t)}^{f}G^T\right)^{\frac{1}{2}} + C^{\frac{1}{2}}\right)^{-1} - \frac{1}{2}{\rm Id}\right)\notag\\
&= \left(K_{\nu_+(t)} - P_tG^TC^{-1}\right)C^{\frac{1}{2}}\left( \left(C+hGP_{\nu_+(t)}^{f}G^T\right)^{\frac{1}{2}} + C^{\frac{1}{2}}\right)^{-1}\notag\\
&\hspace{0.5cm} + \frac{1}{2}P_tG^TC^{-1}\left(C^{\frac{1}{2}} - \left(C+hGP_{\nu_+(t)}^{f}G^T\right)^{\frac{1}{2}}\right)\left(\left(C+hGP_{\nu_+(t)}^{f}G^T\right)^{\frac{1}{2}} + C^{\frac{1}{2}}\right)^{-1}.
\end{align}
Using 
\begin{equation}
\left(C+hGP_k^{f}G^T\right)^{\frac{1}{2}} + C^{\frac{1}{2}} \geq 2C^{\frac{1}{2}} \geq 2\sqrt{\lambda_{\min}(C)}{\rm Id}
\end{equation}
with $\lambda_{\min}(C)$ the smallest eigenvalue of $C$, which yields
\begin{equation}
\left\|\left(C+hGP_k^{f}G^T\right)^{\frac{1}{2}} - C^{\frac{1}{2}}\right\| \leq \frac{1}{2\sqrt{\lambda_{\min}(C)}}\left\|C+hGP_k^{f}G^T - C\right\|= \frac{h}{2\sqrt{\lambda_{\min}(C)}}\left\|GP_k^{f}G^T\right\|,
\end{equation}
together with Lemma \ref{boundPf} and Theorem \ref{covConvergenceESRF} then yields the claim.\\
\noindent
Note that by (\ref{boundW}) it holds
\begin{align}\label{estDiffFor}
&\sum_{i=1}^M \left\|X_{\eta_+(t)}^{(i),f} - X_t^{(i)}\right\|^2 \notag\\
&\lesssim \sum_{i=1}^M \left\|X_{\eta(t)}^{(i),a} - X_t^{(i)}\right\|^2 + h^2\|A\|^2\left(\sum_{i=1}^M \left\|X_{\eta(t)}^{(i),a}\right\|^2\right) + h^2\sqrt{M}(M-1)\kappa
\end{align}
where in expectation, the first $h^2$-term is bounded uniformly in time by Lemma \ref{boundMemAT}. Thus
\begin{align}
&\mathbb{E}\left[\sup_{t\in [0,T]} \sum_{i=1}^M\left\|X_{\eta(t)}^{(i),a} - X_{\eta(t)}^{(i)}\right\|^2\right]\notag\\
&\lesssim \mathbb{E}\left[\sum_{i=1}^M \left\|X_0^{(i),a} - X_0^{(i)}\right\|^2\right] + T\int_0^T \mathbb{E}\left[\sup_{r \in [0,s]}\sum_{i=1}^M\left\|X_{\eta(r)}^{(i),a} - X_r^{(i)}\right\|^2\right]{\rm d}s +O(h^2).
\end{align}
Together with \eqref{contLimXCont} this then proves the claim of Theorem \ref{mainResult} by a Gronwall argument.
\newpage
\subsection{Extension to general Lipschitz-continuous model operators.}\label{SubsectionExt}
In the nonlinear case where
\[ {\rm d}X_t = f\left(X_t\right){\rm d}t + Q^{\frac{1}{2}}{\rm d}W_t\]
yielding a forecast ensemble of the form
\begin{equation}
X_{t_k}^{(i),f} = X_{t_{k-1}}^{(i),a} + hf\left(X_{t_{k-1}}^{(i),a}\right) + hQ^{\frac{1}{2}}\hat{W}_k^{(i),h},
\end{equation}
we do not obtain a recursion of the covariance matrices in closed form and consequently no a priori continuous time limit result as in Theorem \ref{covConvergenceESRF}. Nevertheless, for Lipschitz-continuous $f$ we can show a continuous time limit of the ensemble members which then in turn yields the convergence of the covariance matrices. Observe the following:
\begin{align}\label{diffCovNonlinear}
&\left\|P_{\nu(t)}^{f} - P_t\right\|\notag\\
&= \left\|\frac{1}{M-1}\sum_{i=1}^M \left(X_{\eta(t)}^{(i),f} - \bar{x}_{\nu(t)}^{f}\right)\left(X_{\eta(t)}^{(i),f} - \bar{x}_{\nu(t)}^{f}\right)^T - \left(X_t^{(i)}-\bar{x}_t\right)\left(X_t^{(i)}-\bar{x}_t\right)^T\right\|\notag\\
&\leq 2\left(\left(\mathcal{V}_{\nu(t)}^{f}\right)^{\frac{1}{2}} + \left(\mathcal{V}_t\right)^{\frac{1}{2}}\right)\left(\frac{1}{M-1}\sum_{i=1}^M \left\|X_{\eta(t)}^{(i),f} - X_t^{(i)}\right\|^2\right)^{\frac{1}{2}}
\end{align}
which yields
\begin{equation}
\left\|P_{\nu(t)}^{f} - P_t\right\|^2 \leq 8\left(\mathcal{V}_{\nu(t)}^{f} + \mathcal{V}_t\right)\left(\frac{1}{M-1}\sum_{i=1}^M \left\|X_{\eta(t)}^{(i),f} - X_t^{(i)}\right\|^2\right)
\end{equation}
with $\mathcal{V}_k^{f}$ and $\mathcal{V}_t$ as defined in (\ref{defV}). Similar to \cite{deWiljes2018}, one can show that $\mathcal{V}_t$ is bounded uniformly on $[0,T]$. Since it holds by (\ref{boundW}) that
\begin{align}
\mathcal{V}_k^{f} &= \frac{1}{M-1}\sum_{i=1}^M \left\|X_{t_{k-1}}^{(i),a} - \bar{x}_{k-1}^{a} + h\left(f\left(X_{t_{k-1}}^{(i),a}\right) - \bar{f}_{k-1}^{a}\right) + hQ^{\frac{1}{2}}\hat{W}_k^{(i)}\right\|^2\notag\\
&\leq \left(1+h+2h(Lf)_+ + 8h^2\|f\|_{\text{Lip}}^2\right)\mathcal{V}_{k-1}^{a} + h(1+2h)\sqrt{M}(M-1)\kappa
\end{align}
and
\begin{equation}
\mathcal{V}_k^{a} = \text{tr}(P_k^{a}) = \text{tr}\left(P_k^{f} - hK_kGP_k^{f}\right) = \text{tr}(P_k^{f}) - \left\|P_k^{f}G^TC^{-\frac{1}{2}}\right\|_F^2 \leq \text{tr}(P_k^{f}) = \mathcal{V}_k^{f},
\end{equation}
we thus deduce that both $\mathcal{V}_k^{a}$ and $\mathcal{V}_k^{f}$ are bounded uniformly in $k$ by a Gronwall argument. This further yields that $\left\|P_k^{f}\right\|$ and $\left\|P_k^{a}\right\|$ are bounded uniformly in $k$.\\
Therefore assume existence of a strong solution to
\begin{equation}\label{contFilterNonL}
{\rm d}X_t^{(i)} = f\left(X_t^{(i)}\right){\rm d}t + Q^{\frac{1}{2}}\hat{W}_t^{(i)}{\rm d} t + P_tG^TC^{-1}\left({\rm d}Y_t - \frac{1}{2}G\left(X_t^{(i)}+\bar{x}_t\right){\rm d}t\right)
\end{equation}
with
\begin{equation}
\sup_{t\in [0,T]} \mathbb{E}\left[\left\|X_t^{(i)}\right\|^2\right] < \infty,
\end{equation}
then following similar steps as seen in the proof of Theorem \ref{mainResult} we easily deduce:
\begin{theorem}
Consider Algorithm 1 or Algorithm 2, let $\left(X_t^{(i)}\right)_{t \in [0,T]}$ be the unique strong solution to (\ref{contFilterNonL}) and let $\|P_0^{a}\|$ be bounded uniformly in $\omega$. If
\begin{equation}
\mathbb{E}\left[\sum_{i=1}^M\left\|X_0^{(i),a}-X_0^{(i)}\right\|^2\right] \in O(h),
\end{equation}
then it holds
\begin{equation}
 \mathbb{E}\left[\sup_{t \in [0,T]}\sum_{i=1}^M \left\|X_{\eta(t)}^{(i),a}-X_t^{(i)}\right\|^2\right] \in O(h).
\end{equation}
\end{theorem}

\subsection{Extension to orthogonal transformations.}\label{nonUnique}
\noindent
As pointed out in Section \ref{RemNonUnique}, also post-multiplying with an orthogonal matrix $\mathcal{U}_k$ is a valid transformation where, additionally, ${\bf 1} = (1,...,1)^T$ is an eigenvector of $\mathcal{U}_k$ such that the transformation remains mean-preserving. An example is to be found in \cite{wang2004} where the authors propose a revised version of the ETKF in terms of singular value decompositions: originally in \cite{bishop2001}, the transformation reads
\begin{equation}
T_k = U(I+h\Sigma)^{-\frac{1}{2}}
\end{equation}
as a result of the singular value decomposition
\begin{equation}
\left(E_k^{f}\right)^TG^TC^{-1}GE_k^{f} = U\Sigma V^T.
\end{equation}
This formulation was then modified in \cite{wang2004} via
\begin{equation}
T_k = U(I+h\Sigma)^{-\frac{1}{2}}U^T
\end{equation}
which, apart from other properties as discussed in \cite{wang2004}, gives again a mean-preserving ETKF.\\
\noindent
A natural generalization of the above is to use the transformation
\begin{equation}\label{generalU}
\tilde{T}_k := T_k \mathcal{U}_k, \hspace{1cm} \mathcal{U}_k = \mathcal{U}\left(\mathbb{X}_k^{f}\right),
\end{equation}
where $\mathcal{U}$ is a function of the underlying ensemble $\mathbb{X}_k^{f} = \left(X_{t_k}^{(i),f}\right)_{i=1,...,M}$ taking values in the set of orthogonal matrices that are mean-preserving. As we have already argued in Section \ref{RemNonUnique}, (\ref{generalU}) yields the same covariance matrix as the original transformed ensemble thus Lemma \ref{covConvergenceESRF} carries over immediately.\\
\noindent
Let $E^{a}$ now denote the ensemble resulting from the algorithm using (\ref{generalU}). Writing out the previously analyzed algorithms in this setting gives the following evolution equation
\begin{align}\label{perturbedEvol}
E_k^{a} &= E_{k-1}^{a}\mathcal{U}_k + h\left( \left(A - \frac{1}{2}P_k^{f}G^TC^{-1}G\right)E_{k-1}^{a} + \mathcal{W}_k\right)\mathcal{U}_k + O\left(h^2\right)\\
&= E_{k-1}^{a} + E_{k-1}^{a}\left(\mathcal{U}_k-{\rm Id}\right) + h\left( \left(A - \frac{1}{2}P_k^{f}G^TC^{-1}G\right)E_{k-1}^{a} + \mathcal{W}_k\right)\mathcal{U}_k + O\left(h^2\right).
\end{align}
For the abstract ansatz that, when applied to $E_{k-1}^{a}$, the matrices $\mathcal{U}_k$ evolve according to
\begin{equation}
\mathcal{U}_k = {\rm Id} + h\mathcal{R}_k^{h} + O\left(h^2\right)
\end{equation}
such that $\frac{1}{h}\left(\mathcal{U}_k - {\rm Id}\right)$ converges to some continuous-time process $\mathcal{R}$ in a sense to be specified, a similar analysis of Equation (\ref{perturbedEvol}) on existence of a continuous time limit as in Section \ref{SectionContLimit} should apply. This analysis, however, is beyond the scope of this paper.

\section{Discussion.}\label{SectionDiscuss}\\
\noindent
We want to highlight here the main aspects of Section \ref{SectionContLimit} and especially Theorem \ref{mainResult}: the statement (\ref{senseOfConv}) fully characterizes the continuous time limit of the analyzed filtering algorithms by specifying the sense of convergence as well as the rate of convergence. Interestingly, Theorem \ref{mainResult} gives the same limit result for EAKF, ETKF, and Whitaker, Hamill (2002) all together. This suggests that (\ref{contFilterAT}) forms a universal limiting ensemble in the sense specified by (\ref{senseOfConv}) of the class of ESRF algorithms with deterministic model perturbations and in general fully deterministic EnKF. Indeed, consider for instance the deterministic EnKF in \cite{sakov2008} in which the proposed tranformation of the ensemble deviations $E_k^{f}$ yields (\ref{KFcovA}) with an additional $h^2$-term. Following the analysis in Section \ref{SectionContLimit} together with using deterministic model perturbations satisfying Assumption 1, one can easily show convergence of the ensemble coming from this filter towards solutions of the Ensemble Kalman-Bucy Filter (\ref{contFilterAT}) in the sense of Theorem \ref{mainResult}.\\
\noindent
These results come in handy in the property analysis: in \cite{deWiljes2018}, the authors demonstrated in the fully-observed case (i.e. $G = {\rm Id}$) that (\ref{contFilterAT}) together with deterministic model perturbations of the form (\ref{ReichWCont}), is stable and accurate. Their results easily extend to the case of general deterministic model perturbations fulfilling Assumption 2. Therefore by Theorem \ref{mainResult}, these properties now carry over to the discrete-time counterparts as they are independent of $h$ by construction. This yields the powerful conclusion that by analyzing one continuous-time equation we immediately analyze a whole class of discrete-time algorithms.

\section{Conclusion and Outlook.}\label{SectionConclusion}\\
\noindent
In this paper, we showed the existence of a continuous time limit of a broad class of Ensemble Square Root filtering algorithms with deterministic model perturbations. In the linear setting, we derived general conditions on these perturbation which enabled us to show convergence of the empirical mean and covariance matrix towards their respective counterparts in the Kalman-Bucy Filter in the sense that locally uniformly in time the distance to their continuous-time counterpart decays to zero at rate $h$. Under further assumptions, we showed for three exemplary algorithms the existence of an ensemble solving the Ensemble Kalman-Bucy filtering equations (\ref{contFilterAT}) such that the ensemble-mean-square error between the discrete-time and continuous-time ensemble converges to zero locally uniformly in time in expectation at rate $h$. As shown, this result further holds in the case of nonlinear, Lipschitz-continuous model operators.\\
\noindent
An important general observation coming from this analysis is the universality of the limiting ensemble, i.e. we obtain the same limit for all ESRF and furthermore for all fully deterministic EnKF algorithms.\\
\noindent
Along the above analysis, we identified and discussed suitable assumptions on the deterministic model perturbations to yield the above convergence results. However, these assumptions were motivated by the aim of approximating the Riccati equation (\ref{Ricc}) satisfied by the covariance matrices of the Ensemble Kalman-Bucy Filter which implicitly requires an invertibility condition on the ensemble covariance matrices. We shortly discussed possible generalizations involving projected versions of the Riccati equation, or stochastic perturbations instead.\\
\noindent
The latter will form the next step in conducting continuous time limit analyses for these algorithms. In the setting of \cite{langeStannat2019} where the model operator $f$ and observation operator $g$ were assumed to be Lipschitz-continuous and bounded, we were able to show the existence of a continuous time limit using a forecast step of the form
\begin{equation}\label{forecastStep}
X_{t_k}^{(i),f} = X_{t_{k-1}}^{(i),a} + hf\left(X_{t_{k-1}}^{(i),a}\right) + Q^{\frac{1}{2}}\left(W_{t_k}^{(i)} - W_{t_{k-1}}^{(i)}\right)
\end{equation}
as prescribed by the Euler-Maruyama time-discretizations where $W^{(i)}$ are independent standard Brownian motions. Convergence then holds in the sense that
\begin{equation}
\sup_{t \in [0,T]} \mathbb{E}\left[\sum_{i=1}^M \left\|X_{\eta(t)}^{(i),a} - X_t^{(i)}\right\|^2\right] \in O(h)
\end{equation}
where
\begin{equation}
\begin{aligned}
{\rm d}X_t^{(i)} &= f\left(X_t^{(i)}\right){\rm d}t + Q^{\frac{1}{2}}{\rm d}W_t^{(i)} \\
&\hspace{0.5cm}+\frac{1}{M-1}E_t\mathcal{G}_t^TC^{-1}\left({\rm d}Y_t - \frac{1}{2}\left(g\left(X_t^{(i)}\right) + \bar{g}_t\right){\rm d}t\right)
\end{aligned}
\end{equation}
for the appropriate choice of initial conditions (for notation see \cite{langeStannat2019}). The proof, though, highly relies on the boundedness assumption on $f$ and $g$. Thus extending the analysis to the above setting with $f$ Lipschitz-continuous and $g$ linear is still work in progress.\\
\noindent
In \cite{langeStannat2019} and in this paper, we used the Euler-Maruyama time-discretization to formulate the filtering algorithms. An interesting extension would be to consider different discretization schemes, e.g. implicit Euler or higher-order Taylor expansions, and to analyze resulting limiting equations and further implications on structure or properties of these algorithms. Again, these are considerations for future research.

\bigskip 
\noindent 
{\bf Acknowledgements.} The research of Theresa Lange and Wilhelm Stannat has been partially funded by Deutsche Forschungsgemeinschaft (DFG) - SFB1294/1 - 318763901.

\appendix
\section{Proof of the Lemmas}
\subsection{Proof of Lemma \ref{boundPf}.}\label{Pfbound}
\noindent
First of all, note that it holds for any $k = 1,..., L$ in the sense of symmetric positive semidefinite matrices
\begin{align}\label{PaSmallerPf}
P_k^{a} &= \left({\rm Id} - hK_kG\right)P_k^{f} = P_k^{f} - h P_k^{f}G^T\left(hGP_k^{f}G^T+C\right)^{-1}GP_k^{f}\notag\\
&\leq P_k^{f}.
\end{align}
This yields by Assumption 1
\begin{align}
P_k^{f} &= \left({\rm Id} + hA\right)P_{k-1}^{a}\left({\rm Id} + hA\right)^T + h\left({\rm Id} + hA\right)Q + hQ\left({\rm Id} + hA\right)^T\notag\\
&\hspace{0.5cm} + \frac{h^2}{M-1}\sum_{i=1}^M Q^{\frac{1}{2}}\left(\hat{W}_k^{(i),h}-\hat{w}_k^{h}\right)\left(\hat{W}_k^{(i),h}-\hat{w}_k^{h}\right)^TQ^{\frac{T}{2}}\notag\\
&\leq \left({\rm Id} + hA\right)P_{k-1}^{f}\left({\rm Id} + hA\right)^T + h\left({\rm Id} + hA\right)Q + hQ\left({\rm Id} + hA\right)^T\notag\\
&\hspace{0.5cm} + \frac{h^2}{M-1}\sum_{i=1}^M Q^{\frac{1}{2}}\left(\hat{W}_k^{(i),h}-\hat{w}_k^{h}\right)\left(\hat{W}_k^{(i),h}-\hat{w}_k^{h}\right)^TQ^{\frac{T}{2}}
\end{align}
\noindent
and
\begin{equation}
\begin{aligned}
\left\|P_k^{f}\right\| 
&\leq \left(1 + 2h\|A\| + h^2\|A\|\right)\left\|P_{k-1}^{f}\right\| + 2h\left(1 + h\|A\|\right)\|Q\| + h^2\left\|Q^{\frac{1}{2}}\right\|^2\kappa\\
\end{aligned}
\end{equation}
which with a Gronwall argument yields the claim.

\subsection{Proof of Lemma \ref{BoundMean}.}\label{appProofBoundMean}
\noindent
Since the updated mean satisfies the recursion
\begin{equation}
\bar{x}_{k+1}^{a} = \bar{x}_k^{a} + hA\bar{x}_k^{a} + K_{k+1}\left(\Delta Y_{k+1} - hG\bar{x}_k^{a} - h^2GA\bar{x}_k^{a}\right),
\end{equation}
we can estimate
\begin{align}
\left\|\bar{x}_{k+1}^{a}\right\|^2 &\leq \left\| {\rm Id} + hA - hK_{k+1}G - h^2K_{k+1}GA\right\|^2\left\|\bar{x}_k^{a}\right\|^2 + \left\|K_{k+1}\right\|^2\left\|\Delta Y_{k+1}\right\|^2\notag\\
&\hspace{0.5cm} + 2 \left \langle \left({\rm Id} + hA - hK_{k+1}G - h^2K_{k+1}GA\right)\bar{x}_k^{a}, K_{k+1}\Delta Y_{k+1}\right \rangle.
\end{align}
Observe that by (\ref{formY}) it holds
\begin{align}\label{EmeanObs}
\mathbb{E}\left[\left \langle \bar{x}_k^{a}, K_{k+1}\Delta Y_{k+1}\right \rangle\right] &= \int_{t_k}^{t_{k+1}} \mathbb{E}\left[\left \langle \bar{x}_k^{a}, K_{k+1}GX_s^{\text{ref}}\right \rangle\right] {\rm d}s\notag\\
 &\leq h\mathbb{E}\left[\left\|\bar{x}_k^{a}\right\|^2\right] + h\int_{t_k}^{t_{k+1}}\mathbb{E}\left[\|K_{k+1}\|^2\|G\|^2\left\|X_s^{\text{ref}}\right\|^2\right]{\rm d}s
\end{align}
where due to (\ref{RefTraj}) the last summand is in $O(h^2)$. Thus by boundedness of $\|K_k\|$ uniformly in $k$, we obtain an estimate of the form
\begin{equation}
\mathbb{E}\left[\left\|\bar{x}_{k+1}^{a}\right\|^2\right] \leq \left(1+ hC^{(1)}(h)\right)\mathbb{E}\left[\left\|\bar{x}_k^{a}\right\|^2\right] + hC^{(2)}(h)
\end{equation}
which by a Gronwall argument yields the claim.

\subsection{Proof of Lemma \ref{boundMemAT}.}\label{appProofBoundMem}
\noindent
The recursion
\begin{align}
X_{t_k}^{(i),f} &= X_{t_{k-1}}^{(i),a} + hAX_{t_{k-1}}^{(i),a} + hQ^{\frac{1}{2}}\hat{W}_k^{(i),h},\\
X_{t_k}^{(i),a} &= X_{t_k}^{(i),f} - h\hat{K}_kGX_{t_k}^{(i),f} - h\left(K_k - \hat{K}_k\right)G\bar{x}_k^{f} + K_k \Delta Y_k + \mathcal{R}_k^{h}e_i
\end{align}
together with Lemma \ref{EstSumm} yields an estimate of the following form
\begin{align}
\sum_{i=1}^M \left\|X_{t_{k+1}}^{(i),a}\right\|^2 &\leq (1+h\mathcal{C}_1(h))\sum_{i=1}^M\left\|X_{t_k}^{(i),a}\right\|^2 + h\mathcal{C}_2(h) \notag\\
&\hspace{0.5cm}+ 2M\left\langle \bar{x}_k^{a}, K_{k+1}\Delta Y_{k+1}\right \rangle + MC_3\left\|K_{k+1}\Delta Y_{k+1}\right\|^2.
\end{align}
Using (\ref{RefTraj}) we obtain
\begin{align}
\mathbb{E}\left[\left\|K_{k+1}\Delta Y_{k+1}\right\|^2\right] &\leq 2\mathbb{E}\left[\left\|K_{k+1}\right\|^2\left( \left\|\int_{t_k}^{t_{k+1}}GX_s^{\text{ref}}{\rm d}s\right\|^2 + \left\|C^{\frac{1}{2}}\right\|^2\left\|V_{t_{k+1}} - V_{t_k}\right\|^2\right)\right] \notag\\
&\leq \bar{C}_1h^2 + \bar{C}_2h
\end{align}
for some constants $\bar{C}_1, \bar{C}_2$ independent of $h$ and uniform in $k$. Thus in total this yields with (\ref{EmeanObs})
\begin{equation}
\mathbb{E}\left[\sum_{i=1}^M \left\|X_{t_{k+1}}^{(i),a}\right\|^2\right] \leq \left(1+h\tilde{\mathcal{C}}_1(h)\right)\mathbb{E}\left[\sum_{i=1}^M\left\|X_{t_k}^{(i),a}\right\|^2\right] + h\tilde{\mathcal{C}}_2(h)
\end{equation}
which by a Gronwall argument yields the claim.

\section{Bounds for the modified filter}\label{appProofReichCov}\\
\noindent
The modified filter analyzed in \cite{langeStannat2019} using deterministic model perturbations of the form (\ref{ReichW}) reads as follows:\\
{\bf Algorithm 3.}
\begin{align}
X_{t_k}^{(i),f} &= X_{t_{k-1}}^{(i),a} + hAX_{t_{k-1}}^{(i),a} + \frac{h}{2}Q\left(P_{k-1}^{a}\right)^{-1}\left(X_{t_{k-1}}^{(i),a} - \bar{x}_{k-1}^{a}\right),\\
X_{t_k}^{(i),a} &= X_{t_k}^{(i),f} + K_k\left(\Delta Y_k - \frac{h}{2}G\left(X_{t_k}^{(i),f} + \bar{x}_k^{f}\right)\right).
\end{align}
\noindent
The same analysis as conducted in the main part of this paper also applies for this algorithm in the case of the above particular choice of perturbations due to the following results:
\begin{lemma}\label{covProp}
It holds for each $k = 1, ..., L$:
\begin{itemize}
\item in the sense of symmetric positive semidefinite matrices
\begin{equation}\label{LemmaFirst}
P_k^{a} \leq P_k^{f}
\end{equation}
\item for $h$ small enough there exists a constant $0 < p_T^{*,a} < \infty$ such that
\begin{equation}\label{LemmaThird}
\left\|\left(P_k^{a}\right)^{-1}\right\| \leq p_T^{*,a}
\end{equation}
\item for $h$ small enough there exists a constant $0 < p_T^{*,f} < \infty$ such that
\begin{equation}\label{LemmaForth}
\left\|P_k^{f}\right\| \leq p_T^{*,f}.
\end{equation}
\end{itemize}
\end{lemma}
\begin{proof}
On (\ref{LemmaFirst}): using (\ref{gain}) for $K_k$ we obtain
\begin{align}
P_k^{a} &= \left({\rm Id}-\frac{h}{2}K_kG\right)P_k^{f}\left({\rm Id}-\frac{h}{2}K_kG\right)^T\notag\\
&= P_k^{f} - hP_k^{f}G^T\left(C+hGP_k^{f}G^T\right)^{-1}GP_k^{f} \notag\\
&\hspace{0.5cm}+ \frac{h^2}{4}P_k^{f}G^T\left(C+hGP_k^{f}G^T\right)^{-1}GP_k^{f}G^T\left(C+hGP_k^{f}G^T\right)^{-1}GP_k^{f}\notag\\
&= P_k^{f} - \frac{3}{4}hP_k^{f}G^T\left(C+hGP_k^{f}G^T\right)^{-1}GP_k^{f} \notag\\
&\hspace{0.5cm} - \frac{h}{4}P_k^{f}G^T\left(C+hGP_k^{f}G^T\right)^{-1}C^{-1}\left(C+hGP_k^{f}G^T\right)^{-1}GP_k^{f}\notag\\
&\leq P_k^{f}.
\end{align}
On (\ref{LemmaThird}): observe that by the Woodbury matrix identity it holds
\begin{align}\label{Pinv}
P_k^{a} &= P_k^{f} - \frac{h}{2}K_kGP_k^{f} - \frac{h}{2}P_k^{f}G^TK_k^T + \frac{h^2}{4}K_kGP_k^{f}G^TK_k^T\notag\\
&\geq P_k^{f} - hK_kGP_k^{f}\notag\\
&= P_k^{f} - h P_k^{f}G^T\left(C+hGP_k^{f}G^T\right)^{-1}GP_k^{f}\notag\\
&= \left(\left(P_k^{f}\right)^{-1} + hG^TC^{-1}G\right)^{-1}\\
\Rightarrow \left(P_k^{a}\right)^{-1} &\leq \left(P_k^{f}\right)^{-1} + hG^TC^{-1}G.
\end{align}
Further one can estimate
\begin{align}\label{Pkf}
P_k^{f} &\geq \left({\rm Id}+hA\right)P_{k-1}^{a}\left({\rm Id}+hA\right)^T + hQ + \frac{h^2}{2}AQ + \frac{h^2}{2}QA^T\notag\\
&= \left({\rm Id}+hA\right)\left(P_{k-1}^{a}+\frac{h}{2}Q\right)\left({\rm Id}+hA\right)^T + \frac{h}{2}\left(Q-h^2AQA^T\right).
\end{align}
If
\begin{equation*}
h^2 < \frac{\lambda_{-}(Q)}{\lambda_{+}(Q)\|A\|^2},
\end{equation*}
where $\lambda_{-}$ and $\lambda_{+}$ denote smallest and largest eigenvalue, respectively, then $Q-h^2AQA^T$ is positive semidefinite and by choice of $Q$ it holds
\begin{equation}
\begin{aligned}
P_k^{f} 
&\geq \left({\rm Id}+hA\right)P_{k-1}^{a}\left({\rm Id}+hA\right)^T.
\end{aligned}
\end{equation}
If therefore 
\begin{equation*}
h < \min\left(\frac{1}{\|A\|}, \sqrt{\frac{\lambda_{-}(Q)}{\lambda_{+}(Q)\|A\|^2}}\right) = \sqrt{\frac{\lambda_{-}(Q)}{\lambda_{+}(Q)}} \frac{1}{\|A\|} =: h^*,
\end{equation*}
then
\begin{align}
\left(P_k^{a}\right)^{-1} &\leq \left(P_k^{f}\right)^{-1} + hG^TC^{-1}G\notag\\
&\leq \left({\rm Id}+hA\right)^{-T}\left(P_{k-1}^{a}\right)^{-1}\left({\rm Id}+hA\right)^{-1} + hG^TC^{-1}G\notag\\
&\leq \frac{1}{(1-h\|A\|)^2}\left(P_{k-1}^{a}\right)^{-1}+hG^TC^{-1}G\notag\\
&\leq \frac{1}{(1-h\|A\|)^{2k}}\left(P_0^{a}\right)^{-1} + \left(\sum_{j=0}^{k-1}\frac{1}{(1-h\|A\|)^{2j}}\right)hG^TC^{-1}G.
\end{align}
\noindent
For any $h<h^*$ and any $0\leq j \leq L$ observe that it holds
\[\frac{1}{(1-h\|A\|)^{2j}} \leq e^{2T\frac{\|A\|}{1-h^*\|A\|}} =: \alpha_T,\]
thus
\begin{equation}
\left(P_k^{a}\right)^{-1} \leq \alpha_T\left(P_0^{a}\right)^{-1} + T\alpha_T G^TC^{-1}G
\end{equation}
\noindent
which yields the bound
\begin{equation*}
\left\|\left(P_k^{a}\right)^{-1}\right\| \leq \alpha_T \left\|\left(P_0^{a}\right)^{-1}\right\| + T\alpha_T \|G\|^2\lambda_{+}(C^{-1}) =: p_T^{a,*}.
\end{equation*}
On (\ref{LemmaForth}): it holds by using (\ref{LemmaThird})
\begin{equation}
\begin{aligned}
\left\|P_k^{f}\right\| &\leq \left(1+2h\|A\|+h^2\|A\|^2\right)\left\|P_{k-1}^{a}\right\|+h\|Q\| + O(h^2)
\end{aligned}
\end{equation}
thus by (\ref{LemmaFirst}) and a Gronwall argument we obtain that $\|P_k^{f}\|$ is bounded uniformly in $k$.
\end{proof}
\end{document}